\documentclass[12pt, reqno,oneside]{amsart}
\usepackage{amsthm,amsmath,amssymb, amscd}
\usepackage[a4paper]{geometry}
\usepackage[colorlinks=true,allcolors=blue!80!black]{hyperref}
\usepackage{tikz, tikz-cd}
\usepackage[all]{xy}
\usepackage{subfigure}
\usetikzlibrary{matrix,snakes}
\usepackage{enumerate}
\usepackage{bbm}

\theoremstyle{plain}
\newtheorem{thm}{Theorem}[section]

\newtheorem{prop}[thm]{Proposition}

\newtheorem{proposition}[thm]{Proposition}
\newtheorem{corollary}[thm]{Corollary}

\newtheorem*{thm*}{Theorem}
\newtheorem{abcthm}{Theorem}

\theoremstyle{definition}
\newtheorem{defn}[thm]{Definition}
\newtheorem*{defn*}{Definition}

\newtheorem*{exam*}{Example}

\newtheorem{remark}[thm]{Remark}

\newcommand{\tl}{\mathrm{TL}}
\renewcommand{\t}{\mathbbm{1}}
\newcommand{\C}{\mathbb C}
\newcommand{\F}{\mathcal F}
\newcommand{\Z}{\mathbb Z}
\newcommand{\frakS}{\mathfrak{S}}
\newcommand{\calC}{\mathcal{C}}

\newcommand{\calJ}{\mathcal J}
\newcommand{\ootimes}[1]{\otimes_{\tl_{#1}}}

\renewcommand{\H}{\mathcal{H}}

\DeclareMathOperator{\rk}{rank}

\DeclareMathOperator{\im}{im}

\begin{document}	
	
	\title[Injective words for Temperley-Lieb algebras]{Combinatorics of injective words for Temperley-Lieb algebras}
		\author{Rachael Boyd}
    \address{Max Planck Institute for Mathematics, Bonn}
    \email{rachaelboyd@mpim-bonn.mpg.de}
    \urladdr{https://guests.mpim-bonn.mpg.de/rachaelboyd/} 

    \author{Richard Hepworth}
    \address{Institute of Mathematics, University of Aberdeen}
    \email{r.hepworth@abdn.ac.uk}
    \urladdr{http://homepages.abdn.ac.uk/r.hepworth/pages/} 

    \begin{abstract}
    This paper studies combinatorial properties of the \emph{complex of planar injective words}, a chain complex of modules over the Temperley-Lieb algebra that arose in our work on homological stability. 
    Despite being a linear rather than a discrete object, our chain complex nevertheless exhibits interesting combinatorial properties.
    We show that the Euler characteristic of this complex is the~$n$-th Fine number. We obtain an alternating sum formula for the representation given by its top-dimensional homology module and, under further restrictions on the ground ring, we decompose this module in terms of certain standard Young tableaux.
    This trio of results --- inspired by results of Reiner and Webb for the complex of injective words --- can be viewed as an interpretation of the~$n$-th Fine number as the `planar' or `Dyck path' analogue of the number of derangements of~$n$ letters.
    This interpretation has precursors in the literature, but here emerges naturally from considerations in homological stability.
    Our final result shows a surprising connection between the boundary maps of our complex and the Jacobsthal numbers.
    \end{abstract}

	\maketitle
	
\setcounter{tocdepth}{1}
\tableofcontents

\section{Introduction}

In this work we study combinatorial properties of a highly connected complex that arose in our study of the \emph{Temperley-Lieb algebra} in~\cite{BoydHepworthStability}.
Highly connected complexes arise naturally in many areas of mathematics. {(We will use the word \emph{complex} vaguely: we might mean simplicial complex, poset, chain complex, semi-simplicial set or space, and so on. All of these interpretations have corresponding homology groups, and by \emph{highly connected} we mean that these homology groups vanish except in the top degree.)}
In combinatorics {highly connected complexes occur }as matroid complexes and order complexes of geometric lattices~\cite{Bjorner}, as order complexes of Cohen-Macaulay posets~\cite{BjornerGarsiaStanley}, and in the theory of shellability in its various forms~\cite{Bjorner,BjornerWachs,Kozlov}, to name just a few.
For the authors, highly connected complexes {appear} in the theory of \emph{homological stability}. 
This subject is motivated by the study of homology and cohomology of groups and spaces, and makes extensive use of complexes such as buildings, split buildings, complexes of partial bases (of vector spaces, modules, and free groups), complexes of arcs in surfaces, and many more besides.
Though no standard introductory reference currently exists for homological stability,
we recommend Wahl's paper~\cite{WahlLectures}. 
The introduction of~\cite{RandalWilliamsWahl} may also give a good impression of the theory's scope.

The \emph{complex of injective words} is much studied in both combinatorics and topology.
Its high-connectivity has been proved using various methods, by authors including Farmer~\cite{Farmer}, Maazen~\cite{Maazen}, Bj\"orner-Wachs~\cite{BjornerWachs}, Kerz~\cite{Kerz}, and Randal-Williams~\cite{RandalWilliamsConfig}, and is an important ingredient in proofs of homological stability for the symmetric groups~\cite{Maazen,Kerz,RandalWilliamsConfig}.
Reiner and Webb~\cite{ReinerWebb} studied the complex of injective words from a combinatorial point of view. 
They showed that its Euler characteristic is the number of derangements of~$n$ letters, and they described its top-dimensional homology representation in two ways: as an alternating sum, and in terms {of standard Young tableaux.}
A further decomposition of the top-dimensional homology was given by Hanlon and Hersh in~\cite{HanlonHersh}.    
    
In our work on homological stability for Temperley-Lieb algebras~\cite{BoydHepworthStability}, we introduced and studied the \emph{complex of planar injective words}, a chain complex of modules over the Temperley-Lieb algebra on~$n$ strands,  closely analogous to the (chain complex of the) complex of injective words.
In particular we proved that the homology of our complex is concentrated in degree~$(n-1)$, as is the case for the complex of injective words.

In this paper we study the complex of planar injective words from a combinatorial viewpoint, inspired by the results of Reiner and Webb.
We will see that the role of the{ number of derangements }is now taken by the~$n$-th \emph{Fine number}.
We will also expose an unexpected appearance of the \emph{Jacobsthal numbers}.

\subsection{Temperley-Lieb algebras and planar injective words}\label{subsect:TL intro}
Let~$n\geq 0$, let~$R$ be a commutative ring, and let~$a\in R$.
The \emph{Temperley-Lieb algebra}~$\tl_n(a)$ is the~$R$-algebra with basis given by the planar diagrams on~$n$ strands, taken up to isotopy, and with multiplication given by pasting diagrams and replacing closed loops with factors of~$a$. 
The last sentence was intentionally brief, we hope that its meaning becomes clearer with the following illustration of two elements $x,y\in\tl_5(a)$ 
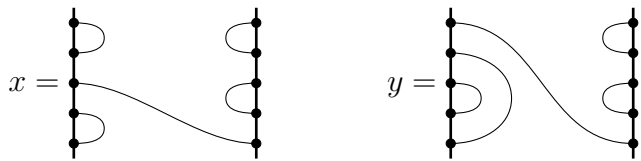
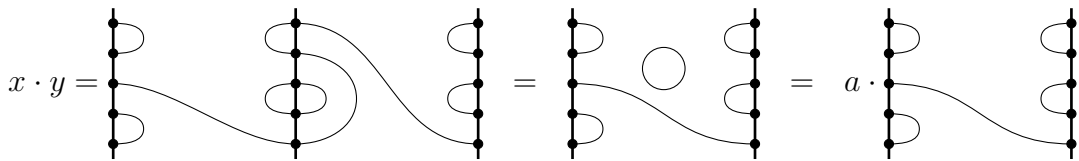
\[
    x=
    \begin{tikzpicture}[scale=0.4, baseline=(base)]
        \coordinate (base) at (0,2.75);
        \draw[line width = 1](0,0.5)--(0,5.5);
        \draw[line width = 1](6,0.5)--(6,5.5);
        \foreach \x in {1,2, 3,4,5}{
            \draw[fill=black] (0,\x) circle [radius=0.15] (6,\x) circle [radius=0.15];
        } 
        \draw (0,1) to[out=0,in=-90] (1,1.5) to[out=90,in=0] (0,2);
        \draw (0,4) to[out=0,in=-90] (1,4.5) to[out=90,in=0] (0,5);
        \draw (6,4) to[out=180,in=-90] (5,4.5) to[out=90,in=180] (6,5);
        \draw (6,2) to[out=180,in=-90] (5,2.5) to[out=90,in=180] (6,3);
        \draw (0,3) .. controls (2,3) and (4,1) .. (6,1);
    \end{tikzpicture}
    \qquad\qquad
    y=
    \begin{tikzpicture}[scale=0.4, baseline=(base)]
        \coordinate (base) at (0,2.75);
        \draw[line width = 1](0,0.5)--(0,5.5);
        \draw[line width = 1](6,0.5)--(6,5.5);
        \foreach \x in {1,2, 3,4,5}{
            \draw[fill=black] (0,\x) circle [radius=0.15] (6,\x) circle [radius=0.15];
        } 
        \draw (0,2) to[out=0,in=-90] (1,2.5) to[out=90,in=0] (0,3);
        \draw (6,4) to[out=180,in=-90] (5,4.5) to[out=90,in=180] (6,5);
        \draw (6,2) to[out=180,in=-90] (5,2.5) to[out=90,in=180] (6,3);
        \draw (0,1) to[out=0,in=-90] (2,2.5) to[out=90,in=0] (0,4);
        \draw (0,5) .. controls (3,5) and (3,1) .. (6,1);
    \end{tikzpicture}
\]
and their product~$x\cdot y$.
\[
    x\cdot y=
    \begin{tikzpicture}[scale=0.4, baseline=(base)]
        \coordinate (base) at (0,2.75);
        \draw[line width = 1](0,0.5)--(0,5.5);
        \draw[line width = 1](6,0.5)--(6,5.5);
        \draw[line width = 1](12,0.5)--(12,5.5);
        \foreach \x in {1,2, 3,4,5}{
            \draw[fill=black] (0,\x) circle [radius=0.15] (6,\x) circle [radius=0.15] (12,\x) circle [radius=0.15];
        } 
        \draw (0,1) to[out=0,in=-90] (1,1.5) to[out=90,in=0] (0,2);
        \draw (0,4) to[out=0,in=-90] (1,4.5) to[out=90,in=0] (0,5);
        \draw (6,4) to[out=180,in=-90] (5,4.5) to[out=90,in=180] (6,5);
        \draw (6,2) to[out=180,in=-90] (5,2.5) to[out=90,in=180] (6,3);
        \draw (0,3) .. controls (2,3) and (4,1) .. (6,1);
        \draw (6,2) to[out=0,in=-90] (7,2.5) to[out=90,in=0] (6,3);
        \draw (12,4) to[out=180,in=-90] (11,4.5) to[out=90,in=180] (12,5);
        \draw (12,2) to[out=180,in=-90] (11,2.5) to[out=90,in=180] (12,3);
        \draw (6,1) to[out=0,in=-90] (8,2.5) to[out=90,in=0] (6,4);
        \draw (6,5) .. controls (9,5) and (9,1) .. (12,1);
    \end{tikzpicture}
    \ \ 
    =
    \ \ 
    \begin{tikzpicture}[scale=0.4, baseline=(base)]
        \coordinate (base) at (0,2.75);
        \draw[line width = 1](0,0.5)--(0,5.5);
        \draw[line width = 1](6,0.5)--(6,5.5);
        \foreach \x in {1,2, 3,4,5}{
            \draw[fill=black] (0,\x) circle [radius=0.15] (6,\x) circle [radius=0.15];
        } 
        \draw (0,1) to[out=0,in=-90] (1,1.5) to[out=90,in=0] (0,2);
        \draw (0,4) to[out=0,in=-90] (1,4.5) to[out=90,in=0] (0,5);
        \draw (0,3) .. controls (3,3) and (3,1) .. (6,1);
        \draw (6,4) to[out=180,in=-90] (5,4.5) to[out=90,in=180] (6,5);
        \draw (6,2) to[out=180,in=-90] (5,2.5) to[out=90,in=180] (6,3);
        \draw (3,3.5) circle (0.7);
    \end{tikzpicture}
    \ \ 
    =
    \ \ 
    a\cdot
    \begin{tikzpicture}[scale=0.4, baseline=(base)]
        \coordinate (base) at (0,2.75);
        \draw[line width = 1](0,0.5)--(0,5.5);
        \draw[line width = 1](6,0.5)--(6,5.5);
        \foreach \x in {1,2, 3,4,5}{
            \draw[fill=black] (0,\x) circle [radius=0.15] (6,\x) circle [radius=0.15];
        } 
        \draw (0,1) to[out=0,in=-90] (1,1.5) to[out=90,in=0] (0,2);
        \draw (0,4) to[out=0,in=-90] (1,4.5) to[out=90,in=0] (0,5);
        \draw (0,3) .. controls (3,3) and (3,1) .. (6,1);
        \draw (6,4) to[out=180,in=-90] (5,4.5) to[out=90,in=180] (6,5);
        \draw (6,2) to[out=180,in=-90] (5,2.5) to[out=90,in=180] (6,3);
    \end{tikzpicture}
\]
The Temperley-Lieb algebras arose in theoretical physics in the 1970s {in work of Temperley and Lieb~\cite{TemperleyLieb}}.  
They were later rediscovered by Jones in his work on von Neumann algebras \cite{JonesIndex}, and used in the first definition of the Jones polynomial \cite{JonesBull}.
Kauffman gave the diagrammatic interpretation of the algebras in \cite{KauffmanState} and \cite{KauffmanInvariant}.
The rank of~$\tl_n(a)$ as an~$R$-module is the~$n$-th Catalan number~$C_n$~\cite{JonesAnnals}.

Now let~$a=v+v^{-1}$ where~$v\in R^\times$ is a unit (the most commonly studied case in the literature).
The \emph{complex of planar injective words}~$W(n)$ is a chain complex of~$\tl_n(a)$-modules. In degree~$i$ it is given by the tensor product module~$\tl_n(a)\otimes_{\tl_{n-i-1}(a)}\t$,
where~$\t$ is the trivial module for~$\tl_{n-i-1}(a)$.
In the original complex of injective words the~$i$-simplices are words~$(x_0,\ldots,x_i)$ on the alphabet~$\{1,\ldots,n\}$ with no repeated entries. 
The action of $\frakS_n$ on these simplices is transitive, and the typical stabiliser is~$\frakS_{n-i-1}$, so that the~$i$-th chain group is isomorphic to~$R\frakS_n\otimes_{R\frakS_{n-i-1}}\t$.
Thus~$W(n)$ is an analogue of (the chain complex of) the complex of injective words, in which the role of~$\frakS_n$ is now played by~$\tl_n(a)$.
In~\cite{BoydHepworthStability} we showed that~$H_d(W(n))=0$ for~$d\leq n-2$, and since the complex is concentrated in degrees from~$-1$ to~$n-1$, it follows that its only homology group is~$H_{n-1}(W(n))$.    
The restriction to the case~$a=v+v^{-1}$ is necessary for~$\tl_n(a)$ to receive a homomorphism from the group algebra of the braid group, which is required in order to define the differentials of~$W(n)$.

\subsection{Results}

The~$n$-th \emph{Fine number}~$F_n$ \cite{FineOEIS} is the number of Dyck paths of length~$2n$ whose first peak has even height. 
This is the second of 11 descriptions of the Fine numbers given by Deutsch and Shapiro in their survey~\cite{DeutschShapiro}.
Deutsch and Shapiro also state the following alternating sum formula for~$F_n$:
\begin{equation}\label{equation-alternating}
    F_n=\frac{1}{n+1}\bigg[ \binom{2n}{n} - 2\binom{2n-1}{n} +3 \binom{2n-2}{n}- \cdots+(-1)^n(n+1)\binom{n}{n} \bigg]
\end{equation}
(See~\cite[Section~4]{DeutschShapiro} and also \cite{Deutsch,Moon, Robertson}.) 
We show that this alternating sum has a very simple interpretation: its~$m$-th term counts the Dyck paths whose first peak has height at least~$m$. {It is surprising that, to our knowledge, this result has not appeared in the literature so far.}

Remarkably, the complex of planar injective words $W(n)$ embodies a representation theoretical `lifting' of the Fine numbers and of this alternating sum formula.

\begin{abcthm}\label{theorem-euler}
    Let $R$ be a commutative ring, let $v\in R^\times$, and let $a = v+v^{-1}$.
    Then the Euler characteristic of~$W(n)$ is the~$n$-th Fine number, up to sign:
    \[
        \chi(W(n)) = (-1)^{n-1}F_n.
    \]
\end{abcthm}

The~$\tl_n(a)$-module~$H_{n-1}(W(n))$ therefore has rank equal to the Fine number~$F_n$.
We call it the \emph{Fineberg module}, and we denote it by~$\F_n(a)$.
(Such top-dimensional homology groups are often called Steinberg modules, after the top-dimensional homology of the Tits building of a vector space.) 
As a consequence of Theorem~\ref{theorem-euler} we obtain the following representation-theoretic lifting of~\eqref{equation-alternating} in terms of induced modules. 
{Recall that the $G$-theory $G_0(A)$ of a Noetherian ring $A$ is the Grothendieck group of the abelian category of finitely generated $A$-modules.  If our ground ring $R$ is Noetherian, then the Temperley-Lieb algebra $\tl_n(a)$ is also Noetherian, so that we may consider $G_0(\tl_n(a))$.} 
\begin{abcthm}\label{corollary-alternating}
    {Suppose that $R$ is a commutative Noetherian ring, that $v\in R^\times$, and that $a=v+v^{-1}$.}  Then the alternating sum formula
    \[
        [\F_n(a)] 
        = 
        \sum_{m=0}^{n}(-1)^m\left[\t\uparrow_{\tl_m(a)}^{\tl_n(a)}\right]
    \]
   holds in $G_0(\tl_n(a))$.
   {(If~$\tl_n(a)$ is semisimple, then $G_0(\tl_n(a))$ is the {group} of virtual representations of~$\tl_n(a)$.)}
\end{abcthm}

We now consider the case~$R=\C$, so that~$a=v+v^{-1}$ with~$v\in\C^\times$.
Then~$\tl_n(a)$ is semisimple unless~$q=v^2$ is an~$\ell$-th root of unity for~$2\leq \ell\leq n$. 
In the case of semisimplicity the irreducible representations~$V_\lambda$ of~$\tl_n(a)$ are indexed by partitions~$\lambda\vdash n$ with at most two columns.
We prove the following description of~$\F_n(a)$ in terms of counts of standard Young tableaux (SYT).

\begin{abcthm}\label{theorem-SYT}
    Let $R=\C$, let~$v\in\C^\times$ be such that~$v^2$ is not an~$\ell$-th root of unity for~$2\leq\ell\leq n$, and let $a=v+v^{-1}$.
    Then
    \[
        \F_n(a)\cong\bigoplus_{\substack{\lambda\vdash n \\ \leq 2\text{ columns}}}|
        \{\text{SYT }Q\text{ of shape }\lambda
        \text{ with top entry of second column odd}\}|\cdot V_\lambda.
    \]
    (In the case~$\lambda=1^n$, the unique SYT of shape~$\lambda$ has no second column, and so we declare that the top entry of its second column is~$(n+1)$.)
\end{abcthm}

The three results listed above are the direct analogues of Reiner and Webb's results relating the complex of injective words to the number of derangements of~$n$ letters~\cite[Propositions~2.1--2.3]{ReinerWebb}.
This suggests an interpretation of the~$n$-th Fine number as the number of `planar derangements' of~$n$ letters.
This interpretation has several precursors in the literature:
One precursor is the fact that the~$n$-th Fine number is equal to the number of Dyck paths of length~$2n$ whose first peak has even height, while D\'esarm\'enien~\cite{Desarmenien} showed that the number of derangements of~$n$ is equal to the number of permutations whose first ascent~$\pi(i)<\pi(i+1)$ occurs for~$i$ even.
Another precursor is that Dyck paths can be interpreted as permutations that avoid the pattern 321~\cite[p.224]{StanleyEnumerativeTwo}, and the Fine number is the number of derangements that avoid 321~\cite[Section~8]{DeutschShapiro}.
It is striking that the same interpretation has arisen naturally through our work on homological stability and injective words.

We now turn to a feature of planar injective words that does not have a precursor in the case of injective words.
The~$n$-th \emph{Jacobsthal number}~$J_n$ \cite{JacobsthalOEIS} is the number of compositions of~$n$ that end with an odd number.
It is equal to the number of sequences~$n>a_1>a_2>\cdots>a_r>0$ whose initial term has the opposite parity to~$n$. 
The~$l^{th}$ \emph{Jacobsthal element} in~$\tl_n(a)$ is defined to be 
\[
    \calJ_l^n
    =
    \sum_{
        \substack{l>a_1>\cdots>a_r>0\\ l-a_1\text{ odd}}
    }
    (-1)^{(r-1)+l} 
    \left(
        \frac{\mu}{\lambda}
    \right)^r
    U_{a_1+n-l}\cdots U_{a_r+n-l}.
\]
(Here~$\lambda$ and~$\mu$ are constants involved in the definition of~$W(n)$ and ~$\frac{\mu}{\lambda}$ is equal to either~$-v$ or~$-v^{-1}$. {The~$U_i$ are the standard generators of $\tl_n(a)$ that we recall in Section~\ref{section-temperley-lieb} below.})
Observe that the number of irreducible terms in~$\calJ_l^n$ is~$J_{l}$. We prove the following identification of the boundary maps of~$W(n)$. 

\begin{abcthm}\label{theorem-jacobsthal}
    Under the assumptions of Theorem~\ref{theorem-euler}, for~$0\leq i\leq n-1$ the boundary map $d^i\colon W(n)_i\rightarrow W(n)_{i-1}$ acts as right multiplication by~$\calJ^n_{i+1}$ in the following sense:
    \[
        d^i(x\otimes r)=x\cdot \calJ^n_{i+1}\otimes r.
    \]
    In particular, the image of the boundary map has number of irreducible terms given by~$J_{i+1}$. 
\end{abcthm}

The formula for the differentials in Theorem~\ref{theorem-jacobsthal} is convenient for explicit computations, and is used in~\cite{BoydHepworthStability} to describe (aspects of) the Fineberg module~$\F_n(a)$ in the case of~$n$ even.

\subsection{Complexes from algebras}

We hope that the results of this paper will encourage others to consider constructing and studying chain complexes of algebra modules from a combinatorial point of view.

The general idea is that one can combine combinatorial complexes with~$\frakS_n$-action (such as the complex of injective words) with finite-dimensional algebras (such as the Temperley-Lieb algebras) and construct algebraic analogues of the complexes.
Examples of possible complexes with~$\frakS_n$-action include the complex of injective words, the realisation of the poset of ordered partitions of~$\{1,\ldots,n\}$ with~$k\geq 2$ parts (this can be naturally identified with the permutahedron), and the realisation of the \emph{partition poset}, which consists of partitions of~$\{1,\ldots,n\}$ with~$1<k<n$ parts. 
Examples of possible algebras include the Temperley-Lieb algebras studied here, Temperley-Lieb algebras of types~$B$ and~$D$, variants such as the dilute and periodic Temperley-Lieb algebras, and cousins such as the Brauer, blob and partition algebras.
(Here we only list examples that are somewhat close to the~$\tl_n(a)$; there will be many other candidates besides.)

In all cases, one can undertake the following `process' that has as input a complex~$C$ with~$\frakS_n$-action and a family of algebras~$A_n$, and as output a chain complex of~$A_n$-modules. The process takes {an}~$\frakS_n$-orbit of simplices, and replaces it with the~$A_n$-module induced from the trivial module modulo the subalgebra corresponding to the stabiliser of the `original' orbit.
(This is of course only a vague process, and its success depends on the nature of~$C$ and the~$A_n$.)

\subsection{Outline}
{In Section \ref{section-homology} we recall the topological basics and results we use, in order to make this paper accessible to combinatorialists.}  In Section~\ref{section-temperley-lieb} we recall the basics of Temperley-Lieb algebras that we require in the rest of the paper.
Section~\ref{section-Wn} recalls the definition of the complex of planar injective words~$W(n)$ from~\cite{BoydHepworthStability}.
Section~\ref{section-dyck-paths} recalls Dyck paths, Dyck words, Catalan numbers and Fine numbers, and refines the usual relationship between them to take into account the height of the first peak, ending with a new account of the alternating sum in Equation~\eqref{equation-alternating}.
In Section~\ref{section-planar-dyck} we recall the relationship between planar diagrams and Dyck words, and prove Theorem~\ref{theorem-euler} and Theorem~\ref{corollary-alternating}.
In Section~\ref{section-young-tableaux} we prove Theorem~\ref{theorem-SYT}.
Finally in Section~\ref{section-jacobsthal} we recall the Jacobsthal numbers and prove Theorem~\ref{theorem-jacobsthal}.

\subsection{Acknowledgements}
The authors would like to thank the Max Planck Institute for Mathematics in Bonn for its support and hospitality, and the anonymous referees for their helpful comments. 

\section{Background on homology} \label{section-homology}

In this section we will give a brief overview of homology, with the aim of making the rest of the paper accessible to combinatorialists. For a more in depth introduction see Weibel~\cite[Section~1.1]{Weibel}, Brown~\cite[Section~I.0]{Brown} or Spanier~\cite[Section~4.1]{Spanier}, which we follow.

Fix an associative ring~$R$.
In what follows, \emph{$R$-module} will mean \emph{left-$R$-module}, unless specified otherwise.  (Note however that we could just as well work with right modules throughout if we wished.)

\begin{defn}
	A \emph{chain complex} of~$R$-modules~$(C,d)$ is a ~$\Z$-graded family of~$R$-modules~$C_n$, $n\in \Z$, and homomorphisms~$d^i:C_i\to C_{i-1}$ such that $d^{i}\circ d^{i+1}=0$ for any~$i\in \Z$. The maps~$d^i$ are called the \emph{differentials} or \emph{boundary operators} or \emph{boundary maps} of the chain complex.
We will often omit explicit mention of the differential, and write a chain complex $(C,d)$ as $C$ alone.
\end{defn}

Given a chain complex~$(C,d)$, the fact that~$d^{i}\circ d^{i+1}=0$ means that the image of~$d^{i+1}$ lies in the kernel of~$d^{i}$. Moreover both of these modules are submodules of~$C_{i}$. This motivates the definition of the $i$-th homology of the chain complex, which measures the failure of exactness (the kernel and image being equal) at~$C_i$.

\begin{defn}
	Given a chain complex~$(C,d)$, the~$i$th \emph{homology module} is the $R$-module given by the subquotient
	\[
	H_i(C)=\ker(d^{i})/\im(d^{i+1}).
	\]
\end{defn}

\begin{defn}
A chain complex is \emph{bounded} if $C_i\neq 0$ for only finitely many degrees $i$.
We will say that the complex is \emph{concentrated in degrees $m\leq i\leq n$} if $C_i=0$ except when $i$ lies in the specified range.
Given a bounded chain complex $C$, the \emph{top degree} is the largest degree $i$ for which $C_i\neq 0$.
A chain complex is \emph{finite} if it is bounded, and each chain module $C_i$ is finitely generated.
\end{defn}

\begin{defn}
	We say that a bounded chain complex~$(C,d)$ is \emph{highly acyclic} or \emph{highly connected} if its homology vanishes in all but the top degree.
\end{defn}

It follows that if~$C_i$ is concentrated in the range~$-1\leq i \leq m$, with $C_m\neq 0$, then it is highly connected if and only if~$H_i(C)=0$ except when~$i=m$. In this case, since the source of~$d^{m+1}$ is~$C_{m+1}=0$, it follows that~$\im(d^{m+1})=0$ and $H_{m}(C)$ is simply $\ker{(d^m)}$, which is a submodule of the top chain module~$C_m$.  Thus there is an exact sequence:
\[
	0\to H_m(C)\to C_m\xrightarrow{d^m} C_{m-1}.
\]

Now suppose that $R$ is equipped with a \emph{rank function}, i.e.~an assigment that sends each finitely generated $R$-module $M$ to an element $\rk(M)$ of some fixed abelian group, with the property that for every short exact sequence 
\[
	0\to A\to B\to C\to 0	
\]
of $R$-modules, we have
\[
	\rk(B) = \rk(A)+\rk(C).
\]
It follows in particular that $\rk(A)=\rk(A')$ whenever $A\cong A'$, and that $\rk(A\oplus B)=\rk(A)+\rk(B)$.  
Examples of rank functions include the case $R=\Z$ with $\rk(M)$ taken to be the usual rank of an abelian group, and the case $R=\mathbbm{k}$ for a field $\mathbbm{k}$ with $\rk(M)$ the vector-space dimension.
The restriction to finitely-generated modules is essential here, because otherwise the \emph{Eilenberg swindle} shows that $\rk(M)=0$ for all $M$, by considering the isomorphism \[M\oplus M^\infty\cong M^\infty.\]

\begin{defn}
	Suppose that $R$ is equipped with a rank function.
	Let~$(C,d)$ be a finite chain complex over $R$. Then we define the \emph{Euler characteristic} of~$(C,d)$ to be the alternating sum
	\[
	\chi(C)=\sum_i (-1)^i \rk(C_i).
	\]
\end{defn}

We wish to compare the Euler characteristic of a chain complex with that of its homology.  In order to do this, we make the assumption that $R$ is Noetherian.
This guarantees that any submodule of a finitely generated $R$-module is again finitely generated, and thus that the homology of a finite chain complex consists of finitely-generated $R$-modules, again concentrated in finitely many degrees.
With these assumptions we have the following.

\begin{proposition}\label{prop-spanier}
	Suppose that $R$ is Noetherian, and that $(C,d)$ is a finite chain complex of $R$-modules.  Then:
\[
	\chi(C)=\sum_i (-1)^i \rk(H_i(C)).
\]
\end{proposition}

This proposition is an example of the \emph{Hopf trace formula} applied to the identity map. A proof in the case $R=\mathbb{Z}$ can be found in Spanier~\cite[Theorem 4.3.14]{Spanier}, and that proof extends to the present setting.

\begin{defn}
Let $R$ be a Noetherian ring.  The group $G_0(R)$ is defined to be the Grothendieck group of the abelian category of finitely generated $R$-modules.
Unpacking this definition, $G_0(R)$ is obtained from the free abelian group generated by the isomorphism classes $[M]$ of finitely-generated $R$-modules $M$, by applying the relations $[B]=[A]+[C]$ whenever we have a short exact sequence
\[
	0\to A\to B \to C \to 0
\]
of finitely-generated $R$-modules.
Then $R$ admits a \emph{universal rank function}, with values in $G_0(R)$, in which the rank of a finitely generated module $M$ is defined to be simply $[M]$.
\end{defn}

Applying Proposition~\ref{prop-spanier} in the case of the universal rank function, we obtain the following:

\begin{prop}\label{prop-G0}
	Let~$R$ be Noetherian, and let $C$ be a finite chain complex over $R$.
	Then the equation
	\begin{equation}\label{equation-Kzero}
		\sum_m (-1)^m[C_m] = \sum_m (-1)^m[H_m(C)]
	\end{equation}
	holds in~$G_0(R)$.
\end{prop}

Note that if $R$ is semisimple, then $G_0(R)$ coincides with the $K$-theory group $K_0(R)$ obtained from finitely generated projective $R$-modules, and that in these cases both coincide with the group of virtual representations of $R$.  (See Example~2.1.4 in Weibels~\cite{WeibelK}.)

\section{Temperley-Lieb algebras}
\label{section-temperley-lieb}

In this section we will cover the basic facts about Temperley-Lieb algebras that we require in the rest of the paper. There is some overlap between the material recalled here and in~\cite{BoydHepworthStability}. 
General references for readers new to the~$\tl_n(a)$ are Section~5.7 of Kassel and Turaev's book~\cite{KasselTuraev} on the braid groups, and especially Sections~1 and~2 of Ridout and Saint-Aubin's survey on the representation theory of the~$\tl_n(a)$~\cite{RidoutStAubin}.

\subsection{Definitions}

A \emph{planar diagram on~$n$ strands} consists of two vertical lines in the plane, decorated with~$n$  dots labelled~$1,\ldots,n$ from bottom to top, together with a collection of~$n$ arcs joining the dots in pairs.  
The arcs must lie between the vertical lines, they must be disjoint, and the diagrams are taken up to isotopy.  
For example, here are two planar diagrams in the case~$n=5$:
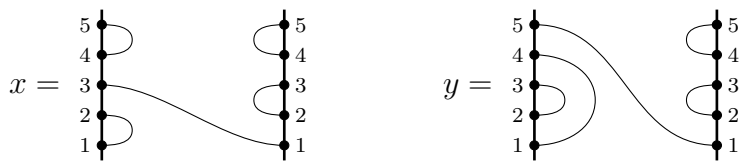
\[
    x=
    \begin{tikzpicture}[scale=0.4, baseline=(base)]
        \coordinate (base) at (0,2.75);
        \draw[line width = 1](0,0.5)--(0,5.5);
        \draw[line width = 1](6,0.5)--(6,5.5);
        \foreach \x in {1,2, 3,4,5}{
            \draw[fill=black] (0,\x) circle [radius=0.15] (6,\x) circle [radius=0.15];
            \draw (0,\x) node[left] {$\scriptstyle{\x}$};
            \draw (6,\x) node[right] {$\scriptstyle{\x}$};
        } 
        \draw (0,1) to[out=0,in=-90] (1,1.5) to[out=90,in=0] (0,2);
        \draw (0,4) to[out=0,in=-90] (1,4.5) to[out=90,in=0] (0,5);
        \draw (6,4) to[out=180,in=-90] (5,4.5) to[out=90,in=180] (6,5);
        \draw (6,2) to[out=180,in=-90] (5,2.5) to[out=90,in=180] (6,3);
        \draw (0,3) .. controls (2,3) and (4,1) .. (6,1);
    \end{tikzpicture}
    \qquad\qquad
    y=
    \begin{tikzpicture}[scale=0.4, baseline=(base)]
        \coordinate (base) at (0,2.75);
        \draw[line width = 1](0,0.5)--(0,5.5);
        \draw[line width = 1](6,0.5)--(6,5.5);
        \foreach \x in {1,2, 3,4,5}{
            \draw[fill=black] (0,\x) circle [radius=0.15] (6,\x) circle [radius=0.15];
            \draw (0,\x) node[left] {$\scriptstyle{\x}$};
            \draw (6,\x) node[right] {$\scriptstyle{\x}$};
        } 
        \draw (0,2) to[out=0,in=-90] (1,2.5) to[out=90,in=0] (0,3);
        \draw (6,4) to[out=180,in=-90] (5,4.5) to[out=90,in=180] (6,5);
        \draw (6,2) to[out=180,in=-90] (5,2.5) to[out=90,in=180] (6,3);
        \draw (0,1) to[out=0,in=-90] (2,2.5) to[out=90,in=0] (0,4);
        \draw (0,5) .. controls (3,5) and (3,1) .. (6,1);
    \end{tikzpicture}
\]
We will often omit the labels on the dots.

\begin{defn}[The Temperley-Lieb algebra~$\tl_n(a)$]
\label{definition-temperley-lieb}
    Let~$R$ be a commutative ring and let~$a\in R$.
    The \emph{Temperley-Lieb algebra}~$\tl_n(a)$ is the~$R$-module with basis given by the planar diagrams on~$n$ strands, and with multiplication defined by placing diagrams side-by-side and joining the ends.
    Any closed loops created by this process are then erased and replaced with a factor of~$a$.
\end{defn}

For example, the product~$xy$ of the elements~$x$ and~$y$ above is {shown in Section~\ref{subsect:TL intro} of the introduction}.

We have subscribed to the heresy of~\cite{RidoutStAubin} by drawing planar diagrams that go from left to right rather than top to bottom.
The identity element of~$\tl_n(a)$ is the planar diagram in which each dot on the left is joined to the corresponding dot on the right by a straight horizontal line.

For~$1\leq i \leq n-1$, we define~$U_i\in\tl_n(a)$ to be the planar diagram shown below.
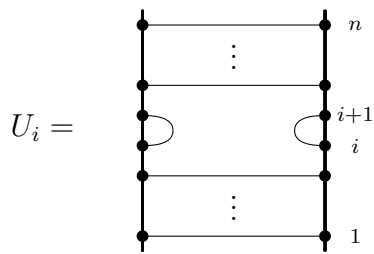
\[
    U_i
    =
    \begin{tikzpicture}[scale=0.4,baseline=(base)]
    \coordinate (base) at (0,4.3);
    \foreach \x in {1, 3,4,5,6,8}
    \foreach \y in {0,6}
    \draw[fill=black, line width=1] (\y,\x) circle [radius=0.15](\y,.5)--(\y,8.5);
    \foreach \x in {1, 3,6,8} 
    \draw (0,\x) --(6,\x);
    \foreach \x in {3}
   \draw (\x,2.2) node {$\scriptstyle{\vdots}$} (\x,7.2) node {$\scriptstyle{\vdots}$};
    \draw
        (7,4) node {$\scriptstyle{i}$}  
        (7,5) node {$\scriptstyle{i+1}$}
        (7,1) node {$\scriptstyle{1}$}
        (7,8) node {$\scriptstyle{n}$};
    \draw (-1,5)node {$\scriptstyle{\phantom{i+1}}$};
    \draw (0,4) to[out=0,in=-90] (1,4.5) to[out=90,in=0] (0,5);
    \draw (6,4) to[out=180,in=-90] (5,4.5) to[out=90,in=180] (6,5);
    \end{tikzpicture}
\]
We refer to an arc joining adjacent dots as a \emph{cup}.
Thus~$U_i$ has a single cup on left and right joining dots~$i$ and~$i+1$.
The elements~$U_i$ satisfy the following relations:
\begin{enumerate}
    \item
   ~$U_iU_j=U_jU_i$ for~$j\neq i\pm 1$
    \item
   ~$U_iU_jU_i = U_i$ for ~$j=i\pm 1$
    \item
   ~$U_i^2 = aU_i$ for all~$i$.
\end{enumerate}
The reader can easily verify these relations for themselves; two of them are shown in Figure~\ref{fig: tl relations}. 
In fact, the generators~$U_i$ together with the three relations above form a presentation of~$\tl_n(a)$ as an~$R$-algebra: Elements of the Temperley-Lieb algebra are formal sums of monomials in the~$U_i$, with coefficients in the ground ring~$R$, modulo the relations above. 
This is proved in \cite[Theorem~2.4]{RidoutStAubin},  \cite[Theorem~5.34]{KasselTuraev}, and \cite[Section~6]{KauffmanDiagrammatic}. We often write~$\tl_n(a)$ as~$\tl_n$. We note here that~$\tl_0=\tl_1=R$.
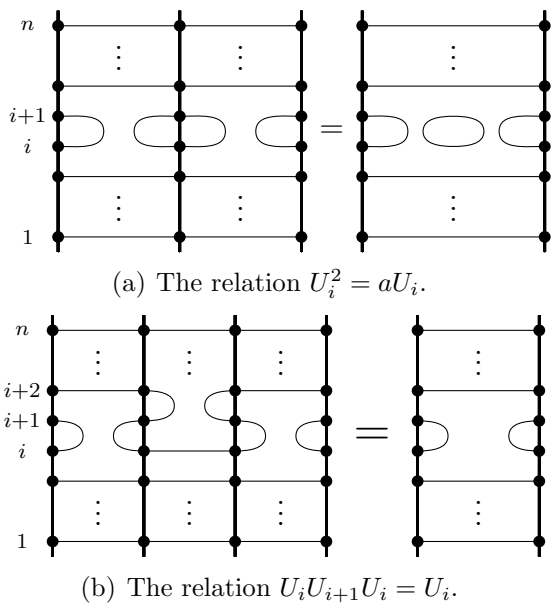
\begin{figure}
    \centering
    \subfigure[The relation~$U_i^2=aU_i$.]{
    \begin{tikzpicture}[scale=0.4]
        \foreach \x in {1, 3,4,5,6,8}
        \foreach \y in {0,4,8,10,16}
        \draw[fill=black, line width=1] (\y,\x) circle [radius=0.15] (\y,.5)--(\y,8.5);
        \foreach \x in {1, 3,6,8}
        \draw[black] (0,\x) --(8,\x) (10,\x)--(16,\x);
        \foreach \x in {2,6,13}
        \draw (\x,2.2) node {$\scriptstyle{\vdots}$} (\x,7.2) node {$\scriptstyle{\vdots}$};
        \draw (-1,1)node {$\scriptstyle{1}$}
        (-1,4) node {$\scriptstyle{i}$}  (-1,5) node {$\scriptstyle{i+1}$}  (-1,8) node {$\scriptstyle{n}$};
        \foreach \x in {0,4,10}
        \draw (\x,4) to[out=0,in=-90] (\x+1.5,4.5) to[out=90,in=0] (\x,5);
        \foreach \x in {4,8,16}
        \draw (\x,4) to[out=180,in=-90] (\x-1.5,4.5) to[out=90,in=180] (\x,5);
        \draw (13,4) to[out=0, in=-90] (14,4.5) to[out=90,in=0] (13,5) to[out=180,in=90] (12,4.5) to[out=-90, in=180] (13,4);
        \draw (9, 4.5) node[scale=1] {$=$};
    \end{tikzpicture}}
    \subfigure[The relation~$U_iU_{i+1}U_i=U_i$.]{
    \begin{tikzpicture}[scale=0.4]
        \foreach \x in {1,3,4,5,6,8}
        \foreach \y in {0,3,6,9,12,16}
        \draw[fill=black, line width=1] (\y,\x) circle [radius=0.15](\y,.5)--(\y,8.5);
        \foreach \x in {1,3,8}
        \draw[black] (0,\x) --(9,\x) (12,\x)--(16,\x);
        \foreach \x in {1.5,4.5,7.5,14}
        \draw (\x,2.2) node {$\scriptstyle{\vdots}$} (\x,7.2) node {$\scriptstyle{\vdots}$};
        \draw (-1,1)node {$\scriptstyle{1}$}
        (-1,4) node {$\scriptstyle{i}$}  (-1,5) node {$\scriptstyle{i+1}$}  (-1,6) node {$\scriptstyle{i+2}$}  (-1,8) node {$\scriptstyle{n}$};
        \foreach \x in {0,6,12}
        \draw (\x,4) to[out=0,in=-90] (\x+1,4.5) to[out=90,in=0] (\x,5);
        \foreach \x in {3,9,16}
        \draw (\x,4) to[out=180,in=-90] (\x-1,4.5) to[out=90,in=180] (\x,5);
        \draw (6,5) to[out=180,in=-90] (5,5.5) to[out=90,in=180] (6,6) -- (9,6) (12,6)--(16,6) (3,4)--(6,4) (0,6)--(3,6) to[out=0,in=90] (4,5.5) to[out=-90, in=0] (3,5);
        \draw (10.5, 4.5) node[scale=1.5] {$=$};
    \end{tikzpicture}
    }
    \caption{Diagrammatic relations in~$\tl_n(a)$.}
    \label{fig: tl relations}
\end{figure}

\subsection{Induced modules}\label{section-induced}

\begin{defn}[The trivial module~$\t$]
    The \emph{trivial} module~$\t$ of the Temperley-Lieb algebra~$\tl_n(a)$ is the module consisting of~$R$ with the action of~$\tl_n(a)$ in which every diagram acts as multiplication by~$0$, except for the identity diagram.
    Equivalently, it is the module on which all of the generators~$U_1,\ldots,U_{n-1}$ act as~$0$.
    We can regard~$\t$ as either a left or right module, and we will usually do that without indicating so in the notation.
\end{defn}	

\begin{defn}[Sub-algebra convention] 
    For~$m\leq n$, we will regard~$\tl_m$ as the sub-algebra of~$\tl_n$ generated by the elements~$U_1,\ldots,U_{m-1}$, or equivalently, the subalgebra in which dots~$m+1,\ldots,n$ on the left are joined to the corresponding dots on the right by horizontal straight lines.
    We will often regard~$\tl_n$ as a left~$\tl_n$-module and a right~$\tl_m$-module, so that we obtain the left~$\tl_n$-module~$\tl_n\otimes_{\tl_m}\t$.
\end{defn}

The induced modules~$\t\uparrow_{\tl_{m}}^{\tl_n}:=\tl_n\otimes_{\tl_m}\t$ will be the building blocks of the complex~$W(n)$. 

A \emph{planar diagram on~$n$ strands with black box of size~$m$} is a planar diagram on~$n$ strands with the dots~$1,\ldots,m$ on the right encapsulated within a \emph{black box}, such that there are no cups with {both} endpoints in the black box.
For example, the planar diagrams with~$4$ strands and black box of size~$3$ are shown below.
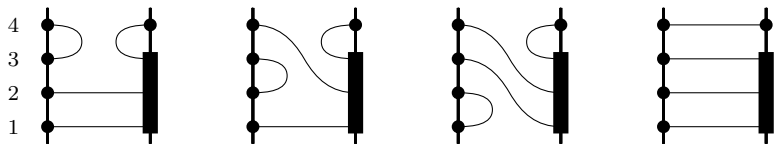
\[
    	\begin{tikzpicture}[scale=0.45]
    	\foreach \x in {1,2,3,4}
    	\draw (-1,\x)node {$\scriptstyle{\x}$}; 
    	\foreach \x in {1,2,3,4}
    	\foreach \y in {0,3,6,9,12,15,18,21}
    	\draw[fill=black, line width=1] (\y,\x) circle [radius=0.15]  (\y,.5)--(\y,4.5);
    	\foreach \x\y in {0/1, 0/2, 6/1, 18/1,18/2,18/3,18/4 }
    	\draw[black] (\x,\y) --(\x+3,\y);
    	\foreach \x\y in {0/3, 6/2, 12/1}
    	\draw (\x,\y) to[out=0,in=-90] (\x+1,\y+.5) to[out=90,in=0] (\x,\y+1);
    	\foreach \x\y in {3/3, 9/3, 15/3}
    	\draw (\x,\y) to[out=180,in=-90] (\x-1,\y+.5) to[out=90,in=180] (\x,\y+1);
    	\foreach \x\y in {6/4, 12/4, 12/3}
    	\draw (\x,\y) to[out=0, in=120] (\x+1.5, \y-1) to[out=300, in=180] (\x+3,\y-2);
    	\foreach \x in {3,9,15,21}
    	\draw[line width=0.2cm] (\x,.8)--(\x,3.2); 
    	\end{tikzpicture}
\]
The~$R$-linear span of the planar diagrams on~$n$ strands with black box of size~$m$ has the structure of a left~$\tl_n(a)$-module.
If~$x$ is a planar diagram on~$n$ strands, and~$y$ is a planar diagram on~$n$ strands with black box of size~$m$, then the product~$x\cdot y$ is defined by pasting the diagrams in the usual way, subject to the condition that if the pasting produces a cup attached to the black box, then the diagram is identified with~$0$.  For example:
\[
    	\begin{tikzpicture}[scale=0.5]
    	\foreach \x in {1,2,3,4}
    	\foreach \y in {3,6,9,12,15,18,21}
    	\draw[fill=black,line width=1]  (\y,\x) circle [radius=0.15]  
    	(\y,.5)--(\y,4.5);
    	\foreach \x\y in {3/1,12/1 }
    	\draw[black] (\x,\y) --(\x+3,\y);
    	\foreach \x\y in { 12/2,3/2, 9/1,9/3, 18/1,18/3}
    	\draw (\x,\y) to[out=0,in=-90] (\x+1,\y+.5) to[out=90,in=0] (\x,\y+1);
    	\foreach \x\y in {15/3,6/3, 12/1,12/3, 21/3,21/1}
    	\draw (\x,\y) to[out=180,in=-90] (\x-1,\y+.5) to[out=90,in=180] (\x,\y+1);
    	\foreach \x\y in {3/4, 12/4}
    	\draw (\x,\y) to[out=0, in=120] (\x+1.5, \y-1) to[out=300, in=180] (\x+3,\y-2);
    	\foreach \x in {6,15,21}
    	\draw[line width=0.2cm] (\x,.8)--(\x,2.2); 
    	\draw (1,2.5) node {$U_1U_3\,\cdot$} (7.5,2.5) node {$=$} (16.5,2.5) node {$=$}(22.5,2.5) node {$=0.$};
    	\end{tikzpicture}
\]

\begin{proposition}\label{proposition-black-box}
    Given~$0\leq m\leq n$,
   ~$\tl_n(a)\otimes_{\tl_m(a)}\t$ is isomorphic to the module of planar diagrams on~$n$ strands with black box of size~$m$.
\end{proposition}

\begin{proof}
    Let~$I_m$ denote the left ideal of~$\tl_n$ generated by the elements~$U_1,\ldots,U_{m-1}$.
    In other words,~$I_m$ is the span of all diagrams which have a cup on the right among dots~$1,\ldots,m$.
    It is shown in~\cite[Lemma 2.12]{BoydHepworthStability} that~$\tl_n\otimes_{\tl_m}\t$ is isomorphic to~$\tl_n/I_m$ under the isomorphism that sends~$x\otimes 1\in\tl_n\otimes_{\tl_m}\t$ to~$x+I_m\in\tl_n/I_m$.
    Since~$\tl_n$ has basis given by planar diagrams, and~$I_m$ has basis given by planar diagrams with cup among dots~$1,\ldots,m$ on the right, the quotient~$\tl_n/I_m$ has basis given by the diagrams with no cups among dots~$1,\ldots,m$ on the right, which we can identify with the~$n$-planar diagrams with black box of size~$m$.
    This determines an~$R$-linear isomorphism between~$\tl_n\otimes_{\tl_m}\t$ and the module of  planar diagrams on~$n$ strands with black box of size~$m$, and it is simple to see that this respects the module structures.
\end{proof}

\subsection{The braiding elements}

Now we suppose that~$a=v+v^{-1}$ where~$v\in R$ is a unit.

\begin{defn}[The braiding elements]\label{definition-si}
    Define~$s_1,\ldots,s_{n-1}\in\tl_n(v+v^{-1})$ by setting
    \[
        s_i=\lambda+\mu U_i
    \]
    where~$\lambda,\mu\in R$ are defined by one of the following {two} options:
    \begin{enumerate}
        \item
       ~$\lambda=-1$ and~$\mu=v$, so that~$s_i=vU_i-1$
        \item
       ~$\lambda=v^2$ and~$\mu=-v$, so that~$s_i=v^2-vU_i$.
    \end{enumerate}
\end{defn}
It is now easy to verify that the elements~$s_i$ satisfy the \emph{braid relations}:
\begin{itemize}
	\item
	$s_i s_j = s_js_i$ for~$i\neq j\pm 1$
	\item
	$s_is_js_i = s_js_is_j$ for~$i=j\pm 1$.
\end{itemize}
Moreover, the~$s_i$ are invertible and satisfy the rule:
\[
s_i^{-1} = \lambda^{-1}+\mu^{-1} U_i.
\]
It is also immediate to verify that $s_i$ acts on~$\t$ as multiplication by~$\lambda$.

The~$s_i$ in fact form the generators in a presentation of~$\tl_n(v+v^{-1})$ as a quotient of the \emph{Iwahori-Hecke algebra} of type~$A_{n-1}$.
In particular, they satisfy further relations of degree~$2$ and~$3$, that we will not list here.
See~\cite{BoydHepworthStability} for more details.

\begin{remark}\label{remark-smoothing}
    There is a homomorphism from (the group algebra of) the braid group into~$\tl_n(v+v^{-1})$ given on generators by~$s_i\mapsto s_i$.  
    This can be regarded as a Kauffman bracket-style smoothing operation from braids to planar diagrams:
    The formula for~$s_i$ tells us to smooth a positive crossing in the two possible ways with weights~$\lambda$ and~$\mu$ as in Figure~\ref{fig: smoothing}, and the formula for~$s_i^{-1}$ tells us to smooth a negative crossing in the two possible ways with weights~$\lambda^{-1}$ and~$\mu^{-1}$.
    In general, given a braid diagram with~$p$ crossings, each crossing is smoothed in the~$2$ possible ways, with appropriate weights, to obtain a linear combination of~$2^p$ planar diagrams.
    We may also consider hybrid diagrams obtained by concatenating planar and braid diagrams, or obtained by partially smoothing braid diagrams.
    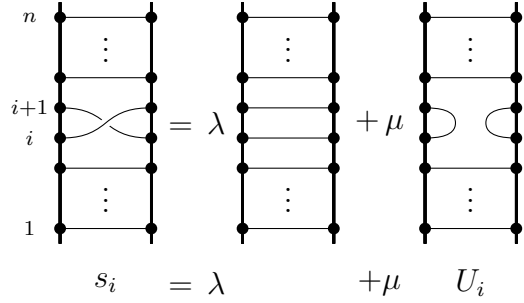
\begin{figure}
        \centering
        	\begin{tikzpicture}[scale=0.4]
        	\foreach \x in {1,3,4,5,6,8}
        	\foreach \y in {0,3,6,9,12,15}
        	\draw[fill=black, line width=1] (\y,\x) circle [radius=0.15] (\y, 0.5)--(\y, 8.5);
        	\foreach \x in {1,3,8}
        	\draw[black] (0,\x) --(3,\x) (6,\x)--(9,\x) (12,\x)--(15,\x);
        	\foreach \x in {1.5,7.5,13.5}
        	\draw (\x,2.2) node {$\scriptstyle{\vdots}$} (\x,7.2) node {$\scriptstyle{\vdots}$};
        	\draw (-1,1)node {$\scriptstyle{1}$} 
        	(-1,4) node {$\scriptstyle{i}$}  (-1,5) node {$\scriptstyle{i+1}$}  (-1,8) node {$\scriptstyle{n}$};
        	\foreach \x in {12}
        	\draw (\x,4) to[out=0,in=-90] (\x+1,4.5) to[out=90,in=0] (\x,5);
        	\foreach \x in {15}
        	\draw (\x,4) to[out=180,in=-90] (\x-1,4.5) to[out=90,in=180] (\x,5);
        	\draw (6,6) -- (9,6) (12,6)--(15,6) (9,4)--(6,4) (0,6)--(3,6) (6,5)--(9,5);
        	\draw (10.5, 4.5) node[scale=1] {$+\,\mu$};
        	\draw (4.5, 4.5) node[scale=1] {$=\, \lambda$};    	
        	\draw (0,5) to[out=0, in=140] (1.5, 4.5) to[out=320, in=180] (3,4);
        	\draw[white,fill=white] (1.5,4.5) circle [radius=0.15];
        	\draw (0,4) to[out=0, in=220] (1.5, 4.5) to[out=40, in=180] (3,5);
        	\draw (1.5,-.8) node[] {$s_i$}(4.5,-.8) node[] {$= \, \lambda$}(10.5,-.8) node[] {$+\mu$}(13.5,-.8) node[] {$U_i$};
        	\end{tikzpicture}
        \caption{Smoothings of~$s_i$.}
        \label{fig: smoothing}
    \end{figure}
\end{remark}

\section{Injective words and planar injective words}
\label{section-Wn}

Throughout this section we will consider the Temperley-Lieb algebra~$\tl_n(a)$, where~$a=v+v^{-1}$ for~$v\in R$ a unit.
We will make use of the elements~$s_1,\ldots,s_{n-1}$ of Definition~\ref{definition-si}.

An \emph{injective word} on the letters~$\{1,\ldots,n\}$ is a tuple~$(x_0,\ldots,x_i)$ whose entries come from the set~$\{1,\ldots,n\}$, with no repeated entries in the tuple.
Injective words form a poset under the subword relation: {$v$ is a subword of $w$, written $w\geq v$, if~$w=(x_0,\ldots,x_i)$ and $v=(x_{k_0},\ldots,x_{k_j})$ for~$0\leq j \leq i$ and $0\leq k_0 <\cdots<k_j\leq i$.}
The {complex of injective words} is most commonly defined as the realisation of this poset, as in~\cite{Farmer} or~\cite{BjornerWachs}.
However, note that for any injective word~$w$, the poset of elements~$v\leq w$ is Boolean. It follows that the poset of injective words is a simplicial poset, and its realisation admits a cell structure in which the cells are simplices, with an~$i$-simplex for each word~$(x_0,\ldots,x_i)$.
This cell complex can be obtained from a semi-simplicial set, as in~\cite{RandalWilliamsConfig}.
For us, the complex of injective words will be the augmented cellular chains of the cell complex described above, studied for example in~\cite{Kerz}.
We define it explicitly now.

\begin{defn}[The complex of injective words]
    The \emph{complex of injective words} is the chain complex~$\calC(n)$ of~$\frakS_n$-modules, concentrated in degrees $-1$ to $(n-1)$, that in degree~$i$ is the free~$R$-module with basis given by tuples~$(x_0,\ldots,x_i)$ where~$x_0,\ldots,x_i\in\{1,\ldots,n\}$ and no letter appears more than once.  
    We allow the empty word~$()$, which lies in degree~$-1$.
    The differential of~$\calC(n)$ sends a word~$(x_0,\ldots,x_i)$ to the alternating sum~$\sum_{j=0}^i(-1)^j(x_0,\ldots,\widehat{x_j},\ldots,x_i)$.
\end{defn}

We can rewrite~$\calC(n)$ in terms of the group algebra~$R\frakS_n$.
Denote by~$s_1,\ldots,s_{n-1}\in\frakS_n$ the adjacent transpositions~$s_i=(i\ \ i+1)$.
These elements satisfy the braid relations listed beneath Definition~\ref{definition-si}.
There is an isomorphism 
\[\calC(n)_i\cong R\frakS_n\otimes_{R\frakS_{n-i-1}}\t,\] where~$\t$ is the trivial module of~$R\frakS_{n-i-1}$. Under this isomorphism the word $(x_0,\ldots,x_i)$ is sent to~$\sigma\otimes 1$ where~$\sigma\in\frakS_n$ is a permutation such that $\sigma(n-i+j) = x_j$.
Furthermore, the differential~$d\colon\calC(n)_i\to\calC(n)_{i-1}$ becomes the map
\[
    d
    \colon 
    R\frakS_n\otimes_{R\frakS_{n-i-1}}\t
    \longrightarrow
    R\frakS_n\otimes_{R\frakS_{n-i}}\t
\]
defined by~$d(x\otimes 1) = \sum_{j=0}^i(-1)^jx\cdot (s_{n-i+j-1}\cdots s_{n-i})\otimes 1$. (See~\cite{HepworthIH}.)
This description inspires the following definition of the planar analogue.

\begin{defn}[The complex of planar injective words \cite{BoydHepworthStability}]\label{defn: W(n) and boundary maps}
    Let~$R$ be a commutative ring, let~$v\in R^\times$, let~$a=v+v^{-1}$, and let~$n\geq 0$.
    The complex of \emph{planar injective words} is the chain complex~$W(n)_\ast$ of~$\tl_n(a)$-modules defined as follows.
    For~$i$ in the range~$-1\leq i\leq n-1$, the degree-$i$ part of~$W(n)_\ast$ is defined by
    \[
        W(n)_i = \tl_n(a)\otimes_{\tl_{n-i-1}(a)}\t
    \]
    and in all other degrees we set~$W(n)_i=0$.
    Note that
    \[W(n)_{-1}=\tl_n(a)\otimes_{\tl_{n}(a)}\t=\t.\]
    For~$i\geq 0$ the boundary map~$d^i\colon W(n)_i\to W(n)_{i-1}$ is defined to be the alternating sum~$\sum_{j=0}^i (-1)^jd^i_j$, where 
    \begin{align*}
        d^i_j\colon\tl_n(a)\otimes_{\tl_{n-i-1}(a)}\t&\rightarrow\tl_n(a)\otimes_{\tl_{n-i}(a)}\t\\
        x\otimes r &\mapsto (x\cdot s_{n-i+j-1}\cdots s_{n-i})\otimes\lambda^{-j}r.
    \end{align*}
    In the expression~$s_{n-i+j-1}\cdots s_{n-i}$, the indices decrease from left to right.
    Observe that~{$d^i_j$} is well-defined because the elements~$s_{n-i},\ldots,s_{n-i+j-1}$ all commute with all generators of~$\tl_{n-i-1}(a)$. {See \cite[Lemma 4.8]{BoydHepworthStability} for the proof that~$d^{i-1}\circ d^i=0$.}
    We have depicted~$W(n)_\ast$ in Figure~\ref{figure-wn}. For notational purposes we will write~$W(n)$ and only use a subscript when identifying a particular degree.
\end{defn}
    
\begin{figure}
 	\[
        \xymatrix{
            \tl_n\ootimes{0}\t
            \ar[d]_{d^{n-1}}
            &
            n-1
            \\
            \tl_n\ootimes{1}\t
            \ar[d]_{d^{n-2}}
            &
            n-2
            \\
            \vdots
            \ar[d]_{d^2}
            &
            {}
            \\
            \tl_n\otimes_{\tl_{n-2}}\t
            \ar[d]_{d^1}
            &
            1
            \\
            \tl_n\otimes_{\tl_{n-1}}\t
            \ar[d]_{d^0}
            &
            0
            \\
            \t
            &
            -1
        }
    \]   
    \caption{The complex~$W(n)$}
    \label{figure-wn}
\end{figure}

\begin{remark}[Visualising~$W(n)$]
    Recall from the diagrammatic description of the induced module~$\tl_n(a)\otimes_{\tl_m(a)}\t$, when~$m\leq n$, given in Section~\ref{section-induced} that elements of~$W(n)_i$ can be regarded as diagrams where the first~$n-i-1$ dots on the right are encapsulated within a black box, and if any cups can be absorbed into the black box, then the diagram is identified with~$0$.
    The differential~$d\colon W(n)_i\to W(n)_{i+1}$ is then given by pasting special elements onto the right of a diagram, followed by taking their signed and weighted sum. These special elements each enlarge the black box by an extra strand, and plumb one of the free strands into the new space in the black box, see Figure~\ref{fig: black box differential}.
\begin{figure}[h!]\label{fig:relative}
    	\begin{tikzpicture}[scale=0.38]
    	\foreach \x in {1,2,3,4}
    	\foreach \y in {0,3,6,9,12,15,18,21,24,27,30}
    	\draw[fill=black, line width=1] (\y,\x) circle [radius=0.15] (\y,0.5)--(\y,4.5);
    	\foreach \x\y in {9/1,9/2,9/3,9/4,18/4}
    	\draw[black] (\x,\y) --(\x+3,\y);
    	\foreach \x\y in {0/3,  6/3, 15/3, 24/3}
    	\draw (\x,\y) to[out=0,in=-90] (\x+1,\y+.5) to[out=90,in=0] (\x,\y+1);
    	\foreach \x\y in {3/1, 9/1, 18/1, 27/1}
    	\draw (\x,\y) to[out=180,in=-90] (\x-1,\y+.5) to[out=90,in=180] (\x,\y+1);
    	\foreach \x\y in {0/1,6/1, 15/1, 24/1,0/2, 6/2, 15/2, 24/2}
    	\draw (\x,\y) to[out=0, in=240] (\x+1.5, \y+1) to[out=60, in=180] (\x+3,\y+2);
    	\foreach \x\y in {27/4}
    	\draw (\x,\y) to[out=0, in=120] (\x+1.5, \y-1) to[out=300, in=180] (\x+3,\y-2);
    	\draw (18,3) to[out=0, in=140] (19.5, 2.5) to[out=320, in=180] (21,2);
    	\foreach \x\y in {19.5/2.5, 28.7/2.65, 28.3/3.35}
    	\draw[white,fill=white] (\x,\y) circle [radius=0.15];
    	\foreach \x\y in {18/2,27/2,27/3}
    	\draw (\x,\y) to[out=0, in=220] (\x+1.5, \y+0.5) to[out=40, in=180] (\x+3,\y+1);
    	\foreach \x in {9,18,27}
    	\draw[line width=0.2cm] (\x-.2,1)--(\x+3,1)--(\x+3,2.2); 
    	\draw[line width=0.2cm] (2.8,1)--(3.2,1);
    	\draw (-1,2.5) node {$d:$} (4.5,2.5) node {$\longmapsto$} (13.5,2.5) node {
    	$\scriptstyle{-\lambda^{-1}}
    	$} (22.5,2.5) node {
    	$\scriptstyle{+\lambda^{-2}}
    	$};
        \draw[snake=brace, mirror snake] (6,-.5)--(12,-.5) node[below,pos=0.5] {$\scriptsize{=0}$};
    	\end{tikzpicture}
    	\caption{Example:~${d^2}\colon W(4)_2\to W(4)_1$}
    	\label{fig: black box differential}
    \end{figure}
    The resulting diagrams can be simplified using the smoothing rules for diagrams with crossings described in Remark~\ref{remark-smoothing}.
    We leave it to the reader to make this description as precise as they wish, and note here that this is where the notion of \emph{braiding}, so often seen in homological stability arguments, fits into our set up.
\end{remark}

In~\cite{BoydHepworthStability} we showed the following analogue of the high-connectivity of the complex of injective words.
It was the main technical underpinning of our proof of homological stability for Temperley-Lieb algebras.

\begin{thm}[{\cite[Theorem E]{BoydHepworthStability}}]\label{theorem-acyclicity}
   ~$H_d(W(n))=0$ for~$d\leq n-2$.
\end{thm}

The top homology of the Tits building is known as the \emph{Steinberg module}. {The rank of the top homology of~$W(n)$ is the~$n$-th Fine number~$F_n$, and} this inspires the name in the following definition.

\begin{defn}
    The \emph{$n$-th Fineberg module} is the~$\tl_n(a)$-module 
    \[\F_n(a) = H_{n-1}(W(n)).\]  We often suppress the~$a$ and simply write~$\F_n$.
\end{defn}

\section{Dyck paths, Catalan numbers, and Fine numbers}
\label{section-dyck-paths}

We now recall Dyck paths, Catalan numbers and Fine numbers.  
We also recall the familiar formula for the Catalan numbers and extend it to take the height of the first peak into account, leading to a new proof of Equation~\eqref{equation-alternating} from the introduction.

A \emph{Dyck path} is a path starting and ending on the horizontal axis, built using the steps~$(1,1)$ and~$(1,-1)$, and never falling below the horizontal axis.  
We abbreviate the steps~$(1,1)$ and~$(1,-1)$ by~$u$ and~$d$ respectively.
Thus Dyck paths are in correspondence with \emph{Dyck words}, i.e.~words in the letters~$u$ and~$d$ containing equal numbers of~$u$s and~$d$s, and such that no initial segment contains more~$d$s than~$u$s.
The following figure shows a Dyck path and its corresponding Dyck word.
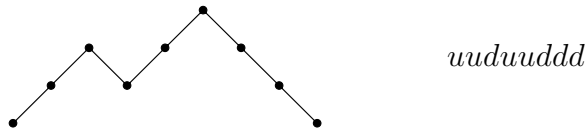
\[
    \begin{tikzpicture}[scale=0.5,baseline=(base)]
        \coordinate (base) at (0,1.5);
        \draw[fill=black]
            (0,0) circle (0.1)
            --(1,1)circle (0.1)
            --(2,2)circle (0.1)
            --(3,1)circle (0.1)
            --(4,2)circle (0.1)
            --(5,3)circle (0.1)
            --(6,2)circle (0.1)
            --(7,1)circle (0.1)
            --(8,0)circle (0.1);
    \end{tikzpicture}
    \qquad\qquad
    uuduuddd
\]
The~$n$-th \emph{Catalan number}~$C_n$ is the number of Dyck paths of length~$2n$.  
For example,~$C_3=5$:
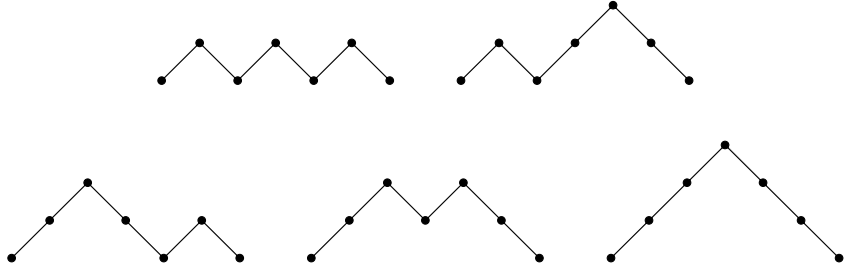
\begin{gather*}
    \begin{tikzpicture}[scale=0.5,baseline=(base)]
        \coordinate (base) at (0,1.5);
        \draw[fill=black]
            (0,0) circle (0.1)
            --(1,1)circle (0.1)
            --(2,0)circle (0.1)
            --(3,1)circle (0.1)
            --(4,0)circle (0.1)
            --(5,1)circle (0.1)
            --(6,0)circle (0.1);
    \end{tikzpicture}
    \qquad
    \begin{tikzpicture}[scale=0.5,baseline=(base)]
        \coordinate (base) at (0,1.5);
        \draw[fill=black]
            (0,0) circle (0.1)
            --(1,1)circle (0.1)
            --(2,0)circle (0.1)
            --(3,1)circle (0.1)
            --(4,2)circle (0.1)
            --(5,1)circle (0.1)
            --(6,0)circle (0.1);
    \end{tikzpicture}
    \\
    \begin{tikzpicture}[scale=0.5,baseline=(base)]
        \coordinate (base) at (0,1.5);
        \draw[fill=black]
            (0,0) circle (0.1)
            --(1,1)circle (0.1)
            --(2,2)circle (0.1)
            --(3,1)circle (0.1)
            --(4,0)circle (0.1)
            --(5,1)circle (0.1)
            --(6,0)circle (0.1);
    \end{tikzpicture}
    \qquad
    \begin{tikzpicture}[scale=0.5,baseline=(base)]
        \coordinate (base) at (0,1.5);
        \draw[fill=black]
            (0,0) circle (0.1)
            --(1,1)circle (0.1)
            --(2,2)circle (0.1)
            --(3,1)circle (0.1)
            --(4,2)circle (0.1)
            --(5,1)circle (0.1)
            --(6,0)circle (0.1);
    \end{tikzpicture}
    \qquad
    \begin{tikzpicture}[scale=0.5,baseline=(base)]
        \coordinate (base) at (0,1.5);
        \draw[fill=black]
            (0,0) circle (0.1)
            --(1,1)circle (0.1)
            --(2,2)circle (0.1)
            --(3,3)circle (0.1)
            --(4,2)circle (0.1)
            --(5,1)circle (0.1)
            --(6,0)circle (0.1);
        \node () at (3,4) {}; 
    \end{tikzpicture}
\end{gather*}
See Corollary~6.2 of Stanley~\cite{StanleyEnumerativeTwo}, and the paragraphs before and after it, for a discussion of the Catalan numbers.
{A \emph{peak} in a Dyck path is a pair of two consecutive steps, the first up and the second down. The \emph{height} of a peak is the $y$-coordinate of the endpoint of the up step.}
The~$n$-th \emph{Fine number}~$F_n$ \cite{FineOEIS} is the number of Dyck paths of length~$2n$ in which the first peak has even height.
For example,~$F_3=2$ as the previous set of diagrams demonstrates.
See Deutsch and Shapiro~\cite{DeutschShapiro} for a nice discussion of the Fine numbers.

\begin{proposition}\label{proposition-binomial}
    The number of Dyck paths of length~$2n$ whose first peak has height~$m$ or greater is
    \[
        \binom{2n-m}{n-m}-\binom{2n-m}{n-m-1}
        =
        \frac{m+1}{n+1}\binom{2n-m}{n}
    \]
    In particular, taking~$m=0$ gives the familiar result
    \[
        C_n = \binom{2n}{n}-\binom{2n}{n+1}=\frac{1}{n+1}\binom{2n}{n}.
    \]
\end{proposition}

This result can be proved using the \emph{reflection trick} commonly attributed to Andr\'e (see Renault's paper~\cite{Renault}), and so is likely to be well-known to combinatorialists.  We include a proof here for completeness and clarity, but we note that it can be extracted from the Ballot Problem presented on page 359 of~\cite{Renault} in the case $a=n+1$, $b=n-m$. 

\begin{proof}
    We begin by recalling the proof of the case~$m=0$ by the `reflection trick'.
    See Lemma~5.27 of~\cite{KasselTuraev} {or \cite{Renault}}.
    We will then adapt this to the general case.
    
    Consider the set of all paths built from the steps~$(1,1)$ and~$(1,-1)$, starting at~$(0,0)$ and ending at~$(2n,0)$.
    The Dyck paths are those that do not go below the~$x$-axis, and the rest we call \emph{bad} paths.
    Given a bad path, we locate the first point at which it meets the line~$y=-1$, and reflect the remainder of the path through that line.
    The result is a path from~$(0,0)$ to~$(2n,-2)$.
    Indeed, this establishes a bijection between the set of bad paths, and the set of paths from~$(0,0)$ to~$(2n,-2)$.
    A path from~$(0,0)$ to~$(2n,0)$ has~$n$ ups and~$n$ downs, so that there are~$\binom{2n}{n}$ in total.
    A path from~$(0,0)$ to~$(2n,-2)$ has~$n-1$ ups and~$n+1$ downs, so there are~$\binom{2n}{n+1}$ in total.
    Therefore the total number of Dyck paths ($C_n$) is~$\binom{2n}{n}-\binom{2n}{n+1}$.

    For general~$m$ we now repeat the procedure, but only consider paths that begin with at least~$m$ up steps.
    Then the number of paths from~$(0,0)$ to~$(2n,0)$ is 
   ~$\binom{2n-m}{n-m}$, and the number from~$(0,0)$ to~$(2n,-2)$ is ~$\binom{2n-m}{n-m-1}$, as we see by considering the distribution of the up moves \emph{after} the first~$m$.
\end{proof}

Now let us fix~$n$.
Given~$0\leq m$, we write~$B_m$ for the  number of Dyck paths whose first peak occurs at height~$m$ or greater.
Thus~$B_m=0$ for~$m>n$.
Then~$(B_m-B_{m+1})$ is the number of Dyck paths whose first peak has height exactly~$m$, and so the Fine number~$F_n$ is nothing other than
\[
    F_n
    =
    (B_0-B_1)+(B_2-B_3)+\cdots
    =
    \sum_{m=0}^n(-1)^m B_m.
\]
In particular, using Proposition~\ref{proposition-binomial} above we recover the formula in Equation~\eqref{equation-alternating}:
\begin{align*}
    F_n
    &=
    \sum_{m=0}^n(-1)^m \frac{m+1}{n+1}\binom{2n-m}{n}
    \\
    &=
    \frac{1}{n+1}\left[ \binom{2n}{n} - 2\binom{2n-1}{n} +3 \binom{2n-2}{n}- \cdots+(-1)^n(n+1)\binom{n}{n} \right]
\end{align*}

\section{Planar diagrams and Dyck paths}
\label{section-planar-dyck}

We now recall the familiar relationship between planar diagrams and Dyck paths, and we extend it to take the height of the first peak into account.

\begin{prop}\label{prop:rank is catalan}
    The set of planar diagrams on~$n$ strands is in bijection with the set of Dyck paths (or words) of length~$2n$.
\end{prop}

\begin{corollary}
    The rank of~$\tl_n(a)$ as an~$R$-module is the Catalan number~$C_n$.
\end{corollary}

There are several choices for such a bijection; the one that is relevant to us is as follows:
Take a planar diagram on~$n$ strands, and work through the dots in order, starting with~$1,\ldots, n$ on the right, followed by~$n,\ldots ,1$ on the left.  
At each dot we encounter an arc, either for the first time or for the second time: if it is the first time, record a~$u$, and if it is the second time, record a~$d$.  
For example, here is a planar diagram, the corresponding Dyck word, and the corresponding Dyck path.
\[
    	\begin{tikzpicture}[scale=0.5,baseline=(base)]
    	    \coordinate (base) at (1,1);
        	\foreach \x in {1,2,3,4}
        	\foreach \y in {3,6}
        	\draw[fill=black,line width=1]  (\y,\x) circle [radius=0.15]
        	(\y,.5)--(\y,4.5);
        	\draw[black] (3,1) --(6,1);
        	\draw (3,2) to[out=0,in=-90] (4,2.5) to[out=90,in=0] (3,3);
        	\draw (6,3) to[out=180,in=-90] (5,3.5) to[out=90,in=180] (6,4);
        	\draw (3,4) to[out=0, in=120] (4.5, 3) to[out=300, in=180] (6,2);
    	\end{tikzpicture}
    	\qquad\qquad
    	uuuddudd
    	\qquad\qquad
    \begin{tikzpicture}[scale=0.5,baseline=(base)]
        \coordinate (base) at (0,0);
        \draw[fill=black]
            (0,0) circle (0.1)
            --(1,1)circle (0.1)
            --(2,2)circle (0.1)
            --(3,3)circle (0.1)
            --(4,2)circle (0.1)
            --(5,1)circle (0.1)
            --(6,2)circle (0.1)
            --(7,1)circle (0.1)
            --(8,0)circle (0.1);
    \end{tikzpicture}
\]
See \cite[pp.966-967]{RidoutStAubin} or~\cite[Lemma~5.33]{KasselTuraev} for details.

\begin{proposition}\label{lemma - rank tensor as number of D paths}
    The rank of~$\tl_n(a)\otimes_{\tl_{m}(a)}\t$ is equal to the number of Dyck paths of length~$2n$ whose first peak occurs at height~$m$ or greater.
\end{proposition}

\begin{proof}
    Proposition~\ref{proposition-black-box} shows that~$\tl_n\otimes_{\tl_m}\t$ has basis given by the~$n$-planar diagrams with black box of size~$m$, i.e.~the diagrams that have no cups among dots~$1,\ldots,m$ on the right.
    These are precisely the diagrams which have no arcs that start and end among dots~$1,\ldots,m$ on the right.
    Therefore, under the bijection between planar diagrams and Dyck paths, these diagrams correspond exactly to the paths that start with~$m$ up steps, i.e.~the paths whose first peak has height~$m$ or greater.
\end{proof}

We are now in a position to prove Theorem~\ref{theorem-euler}, which states that the Euler characteristic of~$W(n)$ is~$(-1)^{n-1}F_n$,
where~$F_n$ is the~$n$-th Fine number.

\begin{proof}[Proof of Theorem~\ref{theorem-euler}]
Let us fix~$n$ and define~$B_m$ to be the number of Dyck paths of length~$2n$ whose first peak occurs at height~$m$ or greater, so that~$\rk(\tl_n\ootimes{m}\t) = B_{m}$.
Let us write~$A_m$ for the number of Dyck paths of length~$2n$ whose first peak occurs at height exactly~$m$. 
Then
\begin{align*}
    \chi(W(n))
    &=
    -\rk(\tl_n\ootimes{n}\t)
    +\rk(\tl_n\ootimes{n-1}\t)
    -\rk(\tl_n\ootimes{n-2}\t)
    +
    \cdots\\
    &\cdots
    +(-1)^{n-2}\rk(\tl_n\ootimes{1}\t)
    +(-1)^{n-1}\rk(\tl_n\ootimes{0}\t)
    \\
    &=
    -B_n
    +B_{n-1}
    -B_{n-2}
    +
    \cdots
    +(-1)^{n-1}B_0
    \\
    &=
    (-1)^{n-1}[
    (B_0-B_1)
    +
    (B_2-B_3)
    +
    \cdots]
\end{align*}
with final term in the bracket either~$B_n$ if~$n$ is even, or~$(B_{n-1}-B_n)$ if~$n$ is odd.  But this is precisely~$(-1)^{n-1}[A_0 + A_2 + A_4 + \cdots + A_n]$ if~$n$ is even, and $(-1)^{n-1}[A_0 + A_2 + A_4 + \cdots + A_{n-1}]$ if~$n$ is odd.  In either case, we obtain~$(-1)^{n-1}F_n$.
\end{proof}

Combined with Proposition~\ref{proposition-binomial}, the proof above gives us Equation~\eqref{equation-alternating}:
\begin{align*}
    F_n
    &=  (-1)^{n-1}\chi(W(n))
    \\
    &= B_0-B_1+\cdots+(-1)^nB_n
    \\
    &= \frac{1}{n+1}\left[\binom{2n}{n}-2\binom{2n-1}{n}+3\binom{2n-2}{n}-\cdots+(-1)^n(n+1)\binom{n}{n}\right]
\end{align*}

We also obtain the representation-theoretic analogue Theorem~\ref{corollary-alternating}.

\begin{proof}[Proof of Theorem~\ref{corollary-alternating}]
	
{Since~$R$ is Noetherian, and~$\tl_n$ is a finitely generated~$R$-algebra, it follows that $\tl_n$ is itself Noetherian, and we can employ Proposition~\ref{prop-G0}, to obtain:}
\begin{align*}
    [\F_n] 
    &= (-1)^{n-1} \sum_{d=-1}^{n-1}(-1)^d[H_d(W(n))]
    \\
    &= (-1)^{n-1}\sum_{d=-1}^{n-1}(-1)^d[W(n)_d]
    \\
    &= (-1)^{n-1}\sum_{d=-1}^{n-1}(-1)^d[\t\uparrow_{\tl_{n-d-1}}^{\tl_n}]
    \\
    &= \sum_{m=0}^{n}(-1)^{m}[\t\uparrow_{\tl_{m}}^{\tl_n}].
\end{align*}
Here the first equation is a consequence of Theorem~\ref{theorem-acyclicity}, the second is an instance of~\eqref{equation-Kzero} in Proposition~\ref{prop-G0}, and the third follows from the definition
\[W(n)_d=\tl_n\otimes_{\tl_{n-d-1}}\t=\t\uparrow^{\tl_n}_{\tl_{n-d-1}}.\qedhere\]

\end{proof}

\section{Young tableaux}
\label{section-young-tableaux}

In this section we will describe the top-dimensional homology~$\F_n=H_{n-1}(W(n))$ as a module over~$\tl_n$ when our ground ring~$R$ is the complex numbers and the algebra~$\tl_n$ is semisimple.   In this case the irreducible representations of~$\tl_n$ are indexed by certain Young diagrams, and we are able to identify the multiplicity of each irreducible in~$\F_n$.  A nice account of the theory used here is given in chapters 4 and 5 of Kassel and Turaev~\cite{KasselTuraev}, see also the brief account in section~11 of Jones' paper~\cite{JonesAnnals}.  In particular we will use the language of partitions, Young diagrams and Young tableaux, for which one can refer to Sections~5.1 and 5.2 of~\cite{KasselTuraev}. 

For this section we will fix~$n\geq 1$ and assume that our ground ring~$R$ is the field of complex numbers~$\C$, that~$v$ and~$a=v+v^{-1}$ are non-zero complex numbers, and that~$q=v^2$ is not a~$d$-th root of unity for~$2\leq d\leq n$.
The latter condition guarantees that~$\tl_n(a)$ is semisimple.
We also assume that~$(\lambda,\mu)=(-1,v)$ in order to accord with the conventions of~\cite{KasselTuraev}.

Under these assumptions, the Temperley-Lieb algebra~$\tl_p(a)$ is semisimple for each~$0\leq p\leq n$, with one irreducible representation~$V_\lambda$ for each partition~$\lambda\vdash p$ whose Young diagram has at most two columns.  The representation corresponding to the partition~$1^p=(1,\ldots,1)\vdash p$, whose Young diagram is a single column of~$p$ boxes, is
\[
    V_{1^p}=\t.
\]
The operations of restriction and induction on the modules~$V_\lambda$ are now determined by the rules
\begin{align*}
    V_\lambda\downarrow^{\tl_p}_{\tl_{p-1}}
    &\cong
    \bigoplus_{\mu\hookrightarrow\lambda} V_\mu,
    \qquad \lambda\vdash p
    \\
    V_\lambda\uparrow^{\tl_{p}}_{\tl_{p-1}}
    &\cong
    \bigoplus_{\lambda\hookrightarrow\mu} V_\mu,
    \qquad\lambda\vdash (p-1)
\end{align*}
where all~$\lambda$ and~$\mu$ are assumed to have diagrams with at most two columns.  Recall that the notation~$\mu\hookrightarrow\lambda$ means that the diagram of~$\mu$ is obtained from that of~$\lambda$ by deleting a single corner box, and that~$\lambda\hookrightarrow\mu$ means that the diagram of~$\mu$ is obtained from that of~$\lambda$ by adding a single corner box.
See the next remark for references to proofs of the facts recalled here.

\begin{remark}[References for the representation theory of Temperley-Lieb algebras]\label{remark-rep-theory}
	With the assumptions from the start of the section, Theorem~2.2 of~\cite{Wenzl} shows that the Iwahori-Hecke algebra~$\H_n(q)$ is semisimple.  Theorem~5.18 of~\cite{KasselTuraev} shows that the distinct irreducible modules of~$\H_n(q)$ are the modules~$V_\lambda$, one for each partition~$\lambda\vdash n$, with no restriction on the shape.  Section 5.7.3 of~\cite{KasselTuraev} then shows that~$\tl_n(a)$ is semisimple, with one irreducible representation~$V_\lambda$ for each partition~$\lambda\vdash n$ of~$n$ whose Young diagram has at most two columns, and that these~$V_\lambda$ pull back to the representations of~$\H_n(q)$ with the same names.  The fact that~$V_{1^n}=\t$ can be seen by comparing Examples~5.12(b) and Theorem~5.29 of~\cite{KasselTuraev}.  The rules for induction and restriction of the representations~$V_\lambda$ of~$\H_n(q)$ are
	\begin{align*}
	V_\lambda\downarrow^{\H_n(q)}_{\H_{n-1}(q)}
	&\cong
	\bigoplus_{\mu\hookrightarrow\lambda} V_\mu,
	\qquad \lambda\vdash n,
	\\
	V_\lambda\uparrow^{\H_{n}(q)}_{\H_{n-1}(q)}
	&\cong
	\bigoplus_{\lambda\hookrightarrow\mu} V_\mu,
	\qquad \lambda\vdash (n-1),
	\end{align*}
	again with no restriction on the shape of~$\lambda$ and~$\mu$.  The first of these rules is Proposition~5.13 of~\cite{KasselTuraev}, and the second follows by Frobenius reciprocity.  From the first of these we can immediately deduce the stated rule for restriction in the Temperley-Lieb case, and the rule for induction then follows, again by Frobenius reciprocity.
	The assumptions on~$n$ stated at the start of this section imply the analogous assumptions for all~$p$ in the range~$0\leq p\leq n$, so that we can replace~$n$ with any such~$p$ throughout this remark.
\end{remark}

\begin{proof}[Proof of Theorem~\ref{theorem-SYT}]
    For this proof, all partitions are assumed to have diagrams with at most two columns.

    To prove the claim it suffices to identify the isomorphism class~$[\F_n]$ within the ring of isomorphism classes of~$\tl_n$-modules.  We have 
    \[
        (-1)^{n-1}[\F_n]
        =
        \sum_{j=-1}^{n-1}(-1)^j[W(n)_j].
    \] 
    {Then since }
    \[
        W(n)_j 
        = 
        \tl_n\otimes_{\tl_{n-j-1}}\t
        =
        V_{1^{n-j-1}}\uparrow_{\tl_{n-j-1}}^{\tl_n}
    \]
   {altogether we have}
    \[
        [\F_n]
        =
        (-1)^{n-1}\sum_{j=-1}^{n-1}(-1)^j[V_{1^{n-j-1}}\uparrow_{\tl_{n-j-1}}^{\tl_n}].
    \]

    We now identify the induced modules appearing above.  Given~$\lambda\vdash n$ and~$p\leq n$, let~$N_{\lambda,p}$ denote the number of SYT for which the labels in the first column begin, starting from the top, with~$1,\ldots, p$. Observe that if~$i$ is in the range~$1\leq i\leq n$, then~$N_{\lambda,i-1}-N_{\lambda,i}$ is precisely the number of SYT of shape~$\lambda$ whose second column has top entry~$i$, and whose first column  necessarily has top entries~$1,\ldots, i-1$ {($N_{\lambda,0}=0$ by convention)}.
    
    We claim:
    \[
        V_{1^p}\uparrow_{\tl_p}^{\tl_n}
        \cong
        \bigoplus_{\lambda\vdash n}V_\lambda^{\oplus N_{\lambda,p}}
    \]
    To see this, we induce~$V_{1^p}$ up to~$\tl_{p+1}$, then~$\tl_{p+2}$, and so on.  At each stage, the module will be a direct sum of modules~$V_\lambda$ for various~$\lambda$, and we will consider each irreducible summand~$V_\lambda$ to be labelled by an SYT of the relevant shape~$\lambda$, according to the following rules.  At the initial step, the single summand~$V_{1^p}$ is labelled by the unique SYT of shape~$1^p$, which is a single column with labels~$1,\ldots,p$. As one passes from one step to the next, we interpret the rule 
    \[
        V_\lambda\uparrow^{\tl_{m+1}}_{\tl_m}
        \cong
        \bigoplus_{\lambda\hookrightarrow\mu} V_\mu
    \]
    as saying that when we induce up the module labelled by an SYT~$Q$, we obtain the sum of the two modules labelled by the SYT obtained from~$Q$ by adding a single box containing~$(m+1)$.  Beginning with the unique SYT of shape~$1^p$, and adding single boxes labelled~$p+1,p+2,\ldots,n$ so that one has an SYT at each stage, produces precisely one copy of each SYT with~$n$ boxes whose first column starts~$1,\ldots,p$.
    Compare with Exercise~5 on p.93 of~\cite{Fulton}.
    This proves the claim.

    The claim above gives us
    \begin{align*}
        [\F_n]
        &=
        (-1)^{n-1}\sum_{j=-1}^{n-1}(-1)^j[V_{1^{n-j-1}}\uparrow_{\tl_{n-j-1}}^{\tl_n}]
        \\
        &=
        (-1)^{n-1}\sum_{j=-1}^{n-1}(-1)^j\sum_{\lambda\vdash n} N_{\lambda,n-j-1}[V_\lambda]
        \\
        &=
        \sum_{\lambda\vdash n}\left[ \sum_{k=0}^{n}(-1)^{k} N_{\lambda,k}\right] [V_\lambda].
    \end{align*}
    Thus the multiplicity of~$V_\lambda$ in~$\F_n$ is~$\sum_{k=0}^{n}(-1)^k N_{\lambda,k}$.  If~$\lambda\neq 1^n$, then this multiplicity is precisely
    \[
        \sum_{i\text{ odd},\ i\leq n} (N_{\lambda,i-1}-N_{\lambda,i})
        =
        |\{\text{SYT of shape }\lambda\text{ with top entry of second column odd}\}|
    \]
    as required.  (The assumption~$\lambda\neq 1^n$ guarantees that~$N_{\lambda,n}=0$, so that a potential final term in the case of~$n$ even does not make a difference.)  If~$\lambda=1^n$ then~$N_{\lambda,i}=1$ for all~$i$, so that the multiplicity is~$0$ if~$n$ is odd and~$1$ if~$n$ is even, which agrees with the special convention outlined in the statement.  This completes the proof.
\end{proof}

\section{Jacobsthal numbers and the boundary maps of \texorpdfstring{$W(n)$}{the complex}}
\label{section-jacobsthal}

In this section we give a combinatorial description of the boundary maps in~$W(n)$ and relate them to the \emph{Jacobsthal numbers}. We start by recalling the boundary maps, and the Jacobsthal numbers.

Recall from Definition \ref{defn: W(n) and boundary maps} that for~$i\geq 0$ the boundary map has the following description
\begin{align*}
d^i\colon W(n)_i\to& W(n)_{i-1}\\
d^i\colon\tl_n\ootimes{n-i-1}\t\to&\tl_n\ootimes{n-i}\t\\
x\otimes r \mapsto& \sum_{j=0}^i (-1)^jd^i_j(x\otimes r)\\
  =& \sum_{j=0}^i(-1)^j(x\cdot s_{n-i+j-1}\cdots s_{n-i})\otimes\lambda^{-j}r\\
  =& \sum_{j=0}^i(-1)^j\lambda^{-j}(x\cdot (\lambda-\mu U_{n-i+j-1})\cdots (\lambda-\mu U_{n-i}))\otimes r.\\
\end{align*}
(Recall that the $U_i$ do not commute in general, so the `descending' ordering of the terms in the products is important.)
Here there are two possibilities for~$\lambda$ and~$\mu$,
namely~$(\lambda,\mu) = (-1,v)$ or $(\lambda,\mu) = (v^2,-v)$, and note for future reference that~$\frac{\mu}{\lambda} = -v$ and~$\frac{\mu}{\lambda}=-v^{-1}$ respectively.

The~$n$-th \emph{Jacobsthal number}~$J_n$~\cite{JacobsthalOEIS} is (among other things) the number of compositions of~$n$ that end with an odd number.  So for example, taking~$n=4$ the relevant compositions are~$31$,~$13$,~$211$,~$121$,~$1111$. 
The Jacobsthal number~$J_n$ can also be described as the number of sequences~$n>a_1>a_2>\cdots>a_r>0$ whose initial term has the opposite parity to~$n$. For the above examples, when~$n=4$, the relevant sequences are~$3$,~$1$,~$3>2$,~$3>1$ and~$3>2>1$. (We allow the empty sequence, and say that by convention its initial term is~$a_1=0$, and~$r=0$. Of course this only occurs when~$n$ is odd.)
The correspondence between compositions and sequences is as follows: Given a composition~$c_1c_2\cdots c_r$, the corresponding sequence is~$n>a_1>\cdots>a_{r-1}>0$ where~$a_j = n-(c_r+c_{r-1}+\cdots+c_{r-j+1})$.  Observe that the initial term is~$a_1=n-c_r$, so that since~$c_r$ is odd,~$a_1$ has the opposite parity to~$n$.

The Jacobsthal numbers are determined by the recursion~$J_n = J_{n-1}+2J_{n-2}$ for~$n\geq 2$, and also satisfy the closed form~$J_n = \frac{2^n-(-1)^n}{3}$.
Thus, the compositions and sequences counted by the Jacobsthal number are about one-third of the total possible sequences and compositions.

\begin{defn}
   Let~$a = v+v^{-1}$ where~$v\in R^\times$ is a unit. For every~$0\leq l\leq n$, we define the~$l^{th}$ \emph{Jacobsthal} element in~$\tl_n(a)$ as follows:
\[
    \calJ_l^n=\sum_{\substack{l>a_1>\cdots>a_r>0\\ l-a_1\text{ odd}}}(-1)^{(r-1)+l}
    \left(\frac{\mu}{\lambda}\right)^rU_{a_1+n-l}\cdots U_{a_r+n-l}
\]
    The indices of the~$U_j$ which occur vary from~$(n-(l-1))$ to~$(n-1)$ and hence are non-trivial in~$\tl_n(a)\otimes_{\tl_{n-(l-1)}(a)}\t$. Recall that we allow the empty sequence ($a_1=0$ and~$r=0$) when~$l$ is odd. This corresponds to a constant summand~$1$ in~$\calJ_l^n$ for odd~$l$. Note that the number of irreducible terms in~$\calJ_l^n$ is~$J_{l}$. {When~$l=n$, we have~$n-l=0$ and the formula simplifies}. We call~$\calJ^{n}_n$ the~\emph{Jacobsthal element}, and denote it~$\calJ_n$.
\end{defn}

\begin{proof}[Proof of Theorem~\ref{theorem-jacobsthal}]
    Firstly we note that the terms appearing in~$\calJ_{i+1}^n$ are non-zero in~$\tl_n\otimes_{\tl_{n-i}}\t${, which is} the target of~$d^i$. We consider the cases~$i$ odd and~$i$ even for clarity.
    For ease of notation, let~$p=n-i-1$. When~$i$ is odd, then~$d^i$ is a sum over an even number of terms, and acts by right multiplication on the left factor of~$x\otimes r$ by the following element:
    \begin{align*}
        & \sum_{j=0}^i (-1)^{j}\lambda^{-j} (\lambda-\mu U_{p+j})\cdots (\lambda-\mu U_{p+1})\\
        =& \sum_{j=0}^{(i-1)/2}\Big(\lambda^{-2j}(\lambda-\mu U_{p+2j})\cdots (\lambda-\mu U_{p+1}) \\
        &\qquad- \lambda^{-2j-1}(\lambda-\mu U_{p+(2j+1)})(\lambda-\mu U_{p+(2j+1)-1})\cdots (\lambda-\mu U_{p+1})\Big)\\
        =& \sum_{j=0}^{(i-1)/2} \Big( \big[\lambda^{-2j} (\lambda-\mu U_{p+2j})\cdots (\lambda-\mu U_{p+1}) \\
        &\qquad- \lambda^{-2j-1}(\lambda)(\lambda-\mu U_{p+(2j+1)-1})\cdots (\lambda-\mu U_{p+1}) \big]\\
        &\qquad \qquad+\lambda^{-2j-1} (\mu U_{p+(2j+1)})(\lambda-\mu U_{p+(2j+1)-1})\cdots (\lambda-\mu U_{p+1})\Big)\\
        =& \sum_{j=0}^{(i-1)/2} \lambda^{-2j-1} \mu U_{p+(2j+1)}(\lambda-\mu U_{p+(2j+1)-1})\cdots (\lambda-\mu U_{p+1})
    \end{align*}
     Here the final equality is given by noting that the terms in the square bracket cancel out. Substituting~$k=2j+1$ gives that~$d^i$ is multiplication by
     \[
         \sum_{\substack{0< k< i+1\\ k\text{ odd}}} \lambda^{-k} \mu U_{p+k}[(\lambda-\mu U_{p+(k-1)})\cdots (\lambda-\mu U_{p+1})]
     \]
     and multiplying out the terms in the square bracket above gives:
    \[
        \sum_{\substack{0< k< i+1\\ k\text{ odd}}} \lambda^{-k} \mu U_{p+k}\Big[\lambda^{k-1}+ \sum_{k>a_2>\ldots>
         a_r>0}\lambda^{k-1-r}(-1)^{r-1}\mu^{r-1} U_{p+a_2}\ldots U_{p+a_r}\Big].
    \]
    Let~$k=a_1$, and note that since~$i$ is odd, then~$k$ being odd equates to~$(i+1)-a_1$ being odd. Putting the two sums in the previous equation together corresponds to the sequences which enumerate the Jacobsthal compositions. Recall that~$p=n-i-1=n-(i+1)$ and since~$i$ is odd then multiplication by~$(-1)^{(i+1)}$ does not change the sign. It follows that~$d^i$ is right multiplication on the left of the tensor product by
    \[
        \calJ^n_{i+1}=\sum_{\substack{i+1>a_1>\cdots>a_r>0\\ (i+1)-a_1\text{ odd}}}(-1)^{(r-1)+(i+1)}\left(\frac{\mu}{\lambda}\right)^rU_{a_1+n-(i+1)}\cdots U_{a_r+n-(i+1)}.
    \]
    
    When~$i$ is even,~$d^i$ is a sum over an odd number of terms, and the element which we left multiply by can be written in a similar fashion to the odd case as follows (once again fixing~$p=n-i-1$).
     
    \begin{align*}
            & \sum_{j=0}^i(-1)^{j} \lambda^{-j}(\lambda-\mu U_{p+j})\cdots (\lambda-\mu U_{p+1})\\
            =& 1 + \sum_{j=1}^i(-1)^{j} \lambda^{-j}(\lambda-\mu U_{p+j})\cdots (\lambda-\mu U_{p+1})\\
            =& 1 - \sum_{j=1}^{i/2} \lambda^{-2j} \mu U_{p+2j}(\lambda-\mu U_{p+(2j-1)})\cdots (\lambda-\mu U_{p+1})).
    \end{align*}
    Substituting~$k=2j$ gives
    \[
        1 - \sum_{\substack{0<k<i+1\\ k\text{ even}}} \lambda^{-k} \mu U_{p+k}(\lambda-\mu U_{p+(k-1)})\cdots (\lambda-\mu U_{p+1}))
    \]
    and we use the computation for~$i$ odd to identify this with~$\calJ_{i+1}^n$. We note that setting~$k=a_1$ gives~$(i+1)-a_1$ odd, since both~$i$ and~$k$ are even, and the negative coefficient of the sum is a consequence of the factor~$(-1)^{(i+1)}$. The constant term~$1$ corresponds to the empty partition, for which we set~$r=0$ and~$a_1=0$.
\end{proof}

We now restrict ourselves to a study of the top differential in this setting. Recall when~$l=n$ we call~$\calJ^{n}_n$ the~\emph{Jacobsthal element}, and denote it~$\calJ_n$. It is a sum of~$J_{n}$ terms.

Since~$\F_n$ is the homology of~$W(n)$ in the top degree, it is simply the kernel of the top differential~$d^{n-1}\colon W(n)_{n-1}\to W(n)_{n-2}$.  There are identifications $W(n)_{n-1}=\tl_n(a)\otimes_{\tl_0(a)}\t\cong\tl_n(a)$
and $W(n)_{n-2}\cong \tl_n(a)\otimes_{\tl_1(a)}\t\cong \tl_n(a)$. 
\begin{prop}
    Under the above identifications, the top differential of~$W(n)$ is right-multiplication by~$\calJ_n$.  In particular, there is an exact sequence
    \[
        0\longrightarrow \F_n(a)
        \longrightarrow \tl_n(a)\xrightarrow{-\cdot\calJ_n}\tl_n(a).
    \]
\end{prop}

\begin{proof}
    This is an application of Theorem~\ref{theorem-jacobsthal} for the case~$i=n-1$, which shows~$d^{n-1}(x\otimes r)=x\cdot \calJ_n \otimes r$. The identifications above send~$x\otimes r$ to~$x\cdot r$ and so under these the map~$d^{n-1}$ is left multiplication by~$\calJ_n$ as described.
\end{proof}

\bibliographystyle{alpha}
\bibliography{tl}
		
\end{document}